\documentclass[12pt,reqno]{amsproc}
\usepackage{ytableau,amsthm,amsmath,amssymb,amsaddr,graphicx}
\usepackage{ytableau,xy,tikz,hyperref,array}
\usepackage[normalem]{ulem}
\xyoption{all}

\newtheorem{theorem}{Theorem}[section]
\newtheorem{lemma}[theorem]{Lemma}

\theoremstyle{definition}
\newtheorem{definition}[theorem]{Definition}

\theoremstyle{remark}
\newtheorem{remark}[theorem]{Remark}

\newtheorem{conjecture}[theorem]{Conjecture}

\DeclareMathOperator{\Par}{Par}

\usepackage{hyperref}

\title[Generators for the Algebra of Symmetric Functions]{Generators for the Algebra\\of Symmetric Functions}
\email{\small{velmurugan@imsc.res.in}}
\author[Velmurugan S]{Velmurugan S}
\address{The Institute of Mathematical Sciences, Chennai}
\address{Homi Bhabha National Institute, Mumbai}
\begin{document}

\begin{abstract}
The algebra of symmetric functions contains several interesting families of symmetric functions indexed by integer partitions or skew partitions.
Given a sequence $\{u_n\}$ of symmetric functions taken from one of these families such that $u_n$ is homogeneous of degree $n$, we provide necessary and sufficient conditions for the sequence to form a system of algebraically independent generators for the algebra of symmetric functions.
\end{abstract}
\maketitle

\section{Introduction}
Let $R$ be a commutative ring with $1$.
Let $\Lambda_{R}$ denote the ring of symmetric functions in infinitely many variables $x_1,x_2,\dotsc$, and coefficients in $R$.
We define a multi-index to be an infinite sequence $\alpha=(\alpha_1,\alpha_2,\dotsc)$ of nonnegative integers with finitely many positive terms.
The size of a multi-index $\alpha$ is the sum of its terms and is denoted by $|\alpha|$.
A multi-index $\alpha$ is said to be a partition if its terms are in weakly decreasing order.
Let $\mathrm{Par}$ denote the set of partitions.
If $\lambda$ is a partition of size $n$, we call it a partition of $n$ and write $\lambda\vdash n$.
The shape of a multi-index is the partition obtained by rearranging its terms in weakly decreasing order.
For each multi-index $\alpha$, define $x^\alpha=\prod_{i\geq1}x_i^{\alpha_i}$.
For each partition $\lambda$, define the monomial symmetric function $m_\lambda=\sum_{\alpha}x^\alpha$, where the sum is over all multi-indices $\alpha$ with shape $\lambda$. 
It is easy to show that $\{m_\lambda \}_{\lambda\in \mathrm{Par}}$ forms an $R$-basis for $\Lambda_R$.
For each integer $n\geq 0$, define the $n$th complete, elementary and power sum symmetric functions as
\begin{displaymath}
    h_n =\sum_{\lambda\vdash n}m_\lambda,\quad e_n = m_{(1^n)}, \text{ and } p_n = m_{(n)},
\end{displaymath}
respectively.

It is well known that $\{h_n\}_{n\geq0}$ and $\{e_n\}_{n\geq0}$ are algebraically independent generators for $\Lambda_{R}$ and $\{p_n\}_{n\geq0}$ forms an algebraically independent generator for $\Lambda_{R}$ if $R$ is a $\mathbb Q$-algebra. Let $\mathbb F$ be a field of characteristic 0.
Then notably, the sequences  $\{m_{(1^n)}\}_{n\geq0}$ and $\{m_{(n)}\}_{n\geq0}$ of monomial symmetric functions are algebraically independent generators for $\Lambda_{\mathbb F}$.
\begin{definition}
A \emph{graded partition sequence} is a sequence $\{\lambda^n\}_{n\geq 0}$ where $\lambda^n$ is a partition of $n$ for every $n\geq 0$.
\end{definition}
D.~G.~Mead~\cite{MR1542321} studied the following natural question:
\begin{quotation}
For which graded partition sequences $\{\lambda^n\}$, does the set $\{m_{\lambda^n}\}_{n\geq0}$ generate $\Lambda_\mathbb{F}$?
\end{quotation}
He obtained the following result.
\begin{theorem}[Mead]
For any graded partition sequence $\{\lambda^n\}$, \linebreak $\{ m_{\lambda^n}\}_{n\geq0}$ generates $\Lambda_\mathbb{F}$.
\end{theorem}

Working with finitely many variables, $x_1,\dotsc,x_m$,
V.~Timofte~\cite{MR2205035} found general criteria for a sequence of symmetric polynomials to freely generate the ring $\Lambda_{m,\mathbb F}$ of symmetric polynomials in $m$ variables.
\begin{theorem}[Timofte]
Let $f_1, f_2, \dots f_m$  be symmetric polynomials in $\Lambda_{m,\mathbb{F}}$ with $\deg f_i\leq i$ for $1\leq i \leq m$.

Then the following are equivalent.
\begin{enumerate}
\item $\mathbb{F}[f_1,\dots,f_m]=\Lambda_{m,\mathbb{F}}$.
\item The matrix $(\partial f_j/\partial x_j)_{1\leq i,j\leq m}$ is non-singular.
\item $\dfrac{\partial f_1}{\partial e_1},\dots,\dfrac{\partial f_m}{\partial e_m}$ are all nonzero in $\Lambda_{m,\mathbb {F}}$.
\end{enumerate}
If the above conditions hold then $f_1,\dots,f_m$ are algebraically independent over $\mathbb{F}$ and deg $f_i=i$ for all $1\leq i\leq m$.
\end{theorem}

In this article, we consider families of symmetric functions indexed by partitions and also skew partitions.
By a skew partition of $n$, we mean a pair $(\lambda,\mu)$ such that $|\lambda|-|\mu|=n$.
Unlike the standard usage (e.g., Macdonald~\cite{MR3443860}), we do not assume that each part of $\lambda$ is at least as large as the corresponding part of $\mu$.
Nonetheless, we denote the skew partition $(\lambda,\mu)$ by $\lambda/\mu$.
\begin{definition}
    A \emph{graded skew partition sequence} is a sequence $\{\lambda^n/{\mu^n}\}_{n\geq 0}$ of skew partitions with $\lambda^n/\mu^n$ a skew partition of $n$ for every $n\geq 0$.
\end{definition}
Given a partition $\lambda=(\lambda_1,\lambda_2,\dotsc)$, define symmetric functions indexed by $\lambda$ as follows:
\begin{displaymath}
    p_\lambda = \prod_{i\geq 1} p_{\lambda_i},\quad h_\lambda = \prod_{i\geq 1} h_{\lambda_i}, \text{ and } e_\lambda = \prod_{i\geq 1} e_{\lambda_i}.
\end{displaymath}
The Hall inner product on $\Lambda_R$ is defined by
\begin{equation}
    \label{eq:hall-inner-prod}
    \langle h_\lambda,m_\mu\rangle = \delta_{\lambda\mu} \text{ for }\lambda,\mu\in\Par.
\end{equation}

The Schur functions can be defined by the Jacobi-Trudi formula
\begin{displaymath}
    s_\lambda = \det(h_{\lambda_i-i+j})_{1\leq i,j\leq n},
\end{displaymath}
where $n$ is the number of non-zero parts of $\lambda$.

In general, given a family $\{u_\lambda\}_{\lambda\in \Par}$ of symmetric functions indexed by integer partitions, we may define skew versions $u_{\lambda/\mu}$ of this family indexed by skew partitions as follows:
\begin{displaymath}
    \langle u_{\lambda/\mu}, f\rangle = \langle u_\lambda, u_\mu f\rangle
\end{displaymath}
for every $f\in \Lambda_{\mathbb F}$.

For each family $\{u_\lambda\}$, we give conditions on a graded partition sequence $\{\lambda^n\}$ (skew partition sequence $\{\lambda^n/\mu^n\}$) such that $\{u_{\lambda^n}\}$ (resp., $\{u_{\lambda^n/\mu^n}\}$) forms a system of algebraically independent generators of $\Lambda_{\mathbb F}$.
Our results are summarized in Table~\ref{table:main}.
\begin{table}
\begin{tabular}{ |>{\raggedright}p{5.5cm}|p{2cm}|>{\raggedright\arraybackslash}p{4cm}|  }
\hline
Symmetric function & Symbol & Condition \\
\hline
monomial   & $m_\lambda$  & no restriction \\
skew monomial   & $m_{\lambda/\mu}$  & $\lambda$ refines $n$ (Defn.~\ref{defn:refinement})\\
 forgotten   & $f_\lambda$  & no restriction \\
skew forgotten   & $f_{\lambda/\mu}$  & $\lambda$ refines $n$\\
skew complete   & $h_{\lambda/\mu}$  & $\lambda_1\geq n$ \\
skew elementary   & $e_{\lambda/\mu}$ & $\lambda_1\geq n$ \\
Schur  &$s_\lambda$ & $\lambda$ is a hook \\
skew Schur   & $s_{\lambda/\mu}$  & $\lambda/\mu$ is a border strip \\
Hall-Littlewood   &$P_\lambda (x;t)$  & no restriction \\
    
    Hall-Littlewood  & $Q_\lambda (x;t)$ & no restriction \\
        big Schur  & $S_\lambda (x;t)$  & $\lambda$ is a hook \\
    
    $Q$-Whittaker   & $W_\lambda (x;t)$  &  no restriction \\
    
    Macdonald   & $P_\lambda (x;q,t)$ &  no restriction \\
    
    integral form of Macdonald \quad   & $J_\lambda (x;q,t)$  & no restriction \\
\hline
\end{tabular}
\smallskip
\caption{Generators for the algebra of symmetric functions}
\label{table:main}
\end{table}
We also investigate specializations of many of these families and determine when they form a system of algebraically independent generators of $\Lambda_{\mathbb F}$.
In our proofs, we shall freely use standard results from the theory of symmetric functions for which the reader is referred to \cite{MR3443860}.
\section{The Lemmas}
We shall start by proving two lemmas which are required in our investigation.
We fix $u_0=1$.
\begin{lemma}
\label{lemma:gen1}
    For each $n\geq 0$, let $u_n\in \Lambda_R$ be homogeneous of degree $n$.
    If $\{u_n\}$ generates $\Lambda_R$ as an $R$-algebra, then $\{u_n\}$ is algebraically independent over $R$.
    The converse holds if $R$ is a field.
\end{lemma}
\begin{proof}
We have a grading $\Lambda_R=\bigoplus_{n\geq0}\Lambda_R^n$, where $\Lambda_R^n$ is the subspace of homogeneous symmetric functions of degree $n$.
Since $\Lambda_R^n$ has basis $\{m_\lambda\mid \lambda\vdash n\}$, it is a free $R$-module of rank $p(n)$, the number of partitions of $n$. For each partition $\lambda \vdash n$, let $u_\lambda = \prod_{i\geq 1}u_{\lambda_i}$.
In $\{u_n\}$ generates $\Lambda_R$ as an $R$-algebra, then $\{u_\lambda\mid\lambda\vdash n\}$ spans $\Lambda^n_R$ as an $R$-module.
It follows that $\{u_\lambda\mid \lambda\vdash n\}$ is linearly independent over $R$, so $\{u_n\}$ is 
algebraically independent over $R$.
Conversely, if $\{u_n\}$ is algebraically independent over $R$, then $\{u_\lambda\mid \lambda\vdash n\}$ is linearly independent over $R$.
If $R$ is a field, it forms a basis for $\Lambda_R^n$, so $\{u_n\}$ generates $\Lambda_R$.
\end{proof}

\begin{lemma}
\label{lemma:gen2} For each $n\geq 0$, let $u_n\in\Lambda_R$ be  homogeneous of degree $n$. Then $\{u_n\}_{n\geq0}$ generates $\Lambda_{R}$ if and only if $\langle u_n,p_n \rangle$ is a unit in $R$ for all $n\geq 0$. 
\end{lemma}
\begin{proof}
    By \eqref{eq:hall-inner-prod}, $\{m_\lambda\mid\lambda\in\Par\}$ and $\{h_\lambda\mid\lambda\in\Par\}$ are dual bases with respect to the Hall inner product, so
    \begin{equation}
        \label{eq:expansion}
        u_n = \langle u_n, m_n\rangle h_n + \sum_{\mu \vdash n,\; \mu\neq (n)} \langle u_n, m_\mu\rangle  h_\mu.
    \end{equation}
    Since $m_n=p_n$, the coefficient of $h_n$ in the above expansion is $\langle u_n,p_n\rangle$.
    By induction on $n$, we may assume that $h_0,\dotsc,h_{n-1}$ lie in the $R$-subalgebra  generated by $\{u_0,\dotsc,u_{n-1}\}$.
    If $\langle u_n,p_n\rangle$ is a unit, $h_n$ lies in the $R$-algebra generated by $\{h_0,\dotsc,h_{n-1},u_n\}$, and hence the $R$-algebra generated by $\{u_0,\dotsc,u_n\}$.
    Since $\{h_n\}_{n\geq 0}$ generates $\Lambda_{R}$ as a $R$-algebra, it follows that $\{u_n\}_{n\geq 0}$ generates $\Lambda_{R}$ if $\langle u_n,p_n\rangle$ is a unit  for all $n\geq 0$.

    Conversely, suppose $\{u_n\}_{n\geq0}$ generates $\Lambda_R$.
    Then, for each $n$, $h_n$ is a polynomial in $u_0,\dotsc,u_n$ with coefficients in $R$.
    Since the sum on the right-hand side of \eqref{eq:expansion} involves only $h_0,\dotsc,h_{n-1}$, we have
    \begin{displaymath}
        u_n = \langle u_n,p_n\rangle f_1(u_0,\dotsc,u_n) + f_2(u_0,\dotsc,u_{n-1})
    \end{displaymath}
    for some polynomials $f_1$ and $f_2$ with coefficients in $R$.
    By Lemma~\ref{lemma:gen1}, $\{u_n\}_{n\geq 0}$ is algebraically independent, allowing us to equate the coefficient of $u_n$ on both sides, implying that $\langle u_n,p_n\rangle$ is a unit.
\end{proof}

\section{Generators for the familiar bases over $\mathbb{F}$}
In this section, we generally work with a field $\mathbb F$ of characteristic zero.
However, $\mathbb F$ always can be replaced by any $\mathbb Q$-algebra $R$.
\subsection{Skew monomial symmetric functions}
We shall generalize Mead's result to the skew case in this section.
\begin{definition}[Domino tabloid, {\cite[Ex.~I.6.5]{MR3443860}}]
\label{def:domino}
Given partitions $\lambda$ and $\mu$, a domino tabloid of shape $\lambda$ and type $\mu$ is a filling of the diagram of $\lambda$ with non-overlapping dominoes of lengths $\mu_1,\mu_2,\dots$, where a domino of length $k$ consists of $k$ consecutive squares in the same row. The weight of a domino tabloid is the product of the lengths of the leftmost domino in each row.
\end{definition}
We will use $\xi\cup \mu$ to denote the partition obtained by taking the parts of $\xi$ and $\mu$ and arranging them in weakly decreasing order. For example, $(2,1)\cup(3,1) = (3,2,1,1)$.
\begin{definition}[Refinement]\label{defn:refinement}
Say that a partition $\lambda=(\lambda_1,\lambda_2,\dotsc)$ refines a positive integer $k$ if there exist $1\leq i_1<i_2<\dotsb$ such that $\lambda_{i_1}+\lambda_{i_2}+\dotsb = k$.
\end{definition}
\begin{theorem}
Let $\{\lambda^n/\mu^n\}$ be a graded skew partition sequence, $u_n=m_{\lambda^n/\mu^n}$, then the set $\{u_n\}_{n\geq0}$ is algebraically independent and generates $\Lambda_{\mathbb F}$ if and only if, for each $n\geq1$, $\lambda^n$ refines $n$.
\label{theorem:mlamu}
\end{theorem}
\begin{proof}By Lemmas~\ref{lemma:gen1} and ~\ref{lemma:gen2}, it is sufficient to prove that\linebreak $\langle m_{\lambda^n/\mu^n},p_n \rangle \neq 0$ if and only if $\lambda^n$ refines $n$ (to simplify notation, let us call $\lambda^n$ as $\lambda$ and $\mu^n$ as $\mu$). Since
$$
p_\lambda = \sum_\mu \epsilon_\mu \epsilon_\lambda w_{\lambda \mu} h_\mu,
$$
where, for any partition $\lambda$, $\epsilon_\lambda$ denotes the sign of a permutation with cycle type $\lambda$, and $w_{\lambda \mu}$ denotes the sum of the weights of all domino tabloids of shape $\lambda$ and type $\mu$ (see \cite[Ex.~I.6.8]{MR3443860}). Therefore we have
\begin{equation}
\label{equation-sm:inn}
\langle m_{\lambda},p_\mu \rangle = \epsilon_\mu \epsilon_\lambda w_{\mu\lambda }
\end{equation}
by the duality of $\{m_\lambda\}_{\lambda\in\Par}$ and $\{h_\lambda\}_{\lambda\in \Par}$ with respect to the Hall inner product.
Hence, $m_\lambda=\sum\limits_{\nu}  \epsilon_\nu \epsilon_\lambda z_{\nu}^{-1} w_{\nu\lambda} p_\nu.$  

Now 
\begin{align*}
\langle m_{\lambda/\mu},p_n \rangle & = \langle m_{\lambda},m_\mu p_n \rangle\\
 & =\Big\langle \sum_{\nu\vdash|\lambda|}  \epsilon_\nu \epsilon_\lambda z_{\nu}^{-1} w_{ \nu \lambda} p_\nu ,\sum_{\xi\vdash |\mu|} \epsilon_\xi \epsilon_\mu z_{\xi}^{-1} w_{\xi \mu} p_{\xi \cup (n)} \Big\rangle\\
 &=\epsilon_\mu\epsilon_\lambda \sum_{\xi\vdash|\mu|} \epsilon_{\xi\cup (n)}\epsilon_\xi z_{\xi}^{-1} z_{\xi\cup(n)}^{-1} w_{\xi \mu} w_{(\xi\cup (n)) \lambda} z_{\xi\cup(n)}\\
 &=\epsilon_\mu\epsilon_\lambda \sum_{\xi\vdash|\mu|} \epsilon_{\xi\cup (n)}\epsilon_\xi z_{\xi}^{-1} w_{\xi \mu} w_{(\xi\cup (n)) \lambda}
\end{align*}
Since $\epsilon_{\xi\cup(n)}=(-1)^{n-1}\epsilon_\xi$,
\begin{equation}
\label{eq:skew-mono}
    \langle m_{\lambda/\mu},p_n \rangle= (-1)^{n-1}\epsilon_\mu\epsilon_\lambda \sum_{\xi\vdash|\mu|} z_{\xi}^{-1} w_{\xi \mu} w_{(\xi\cup (n)) \lambda}.
\end{equation}
Since the terms of the sum above are all positive, the inner product is nonzero if and only if the sum has at least one non-zero term.
That happens precisely when $\lambda$ refines $n$, as we can always take $\xi=(|\mu|)$.
\end{proof}
Taking $\mu^n=\emptyset $, for all $n$, we recover Mead's result as the condition in Theorem~\ref{theorem:mlamu} is trivially satisfied.

An interesting fact is that $\sum_{\xi\vdash|\mu|} z_{\xi}^{-1} w_{\xi \mu} w_{(\xi\cup (n)) \lambda}$, a priori a rational number, is always a nonnegative integer.
Finding a combinatorial interpretation of this number seems to be an interesting question.
In Theorem~\ref{theorem:SM} we characterize the pairs $(\lambda,\mu)$ for which this sum is equal to one.
 
Not all algebraically independent sequences $\{u_n\}$ of symmetric functions in $\Lambda_{\mathbb Z}$ generate $\Lambda_{\mathbb Z}$ as a $\mathbb Z$-algebra.
A well-known example is the sequence $\{p_n\}_{n\geq0}$ of power sum symmetric functions.
The following theorem refines Mead's result to $\Lambda_{\mathbb Z}$.
\begin{theorem}
Let $\{\lambda^n\}$ be a graded partition sequence and for each $n\geq0$, $u_n=m_{\lambda^n}$. Then the set $\{u_n\}_{n\geq0}$ is algebraically independent and generates $\Lambda_\mathbb{Z}$ if and only if, for each $n\geq0$, $\lambda^n=(1^n)$.
\end{theorem}
\begin{proof}
From \eqref{equation-sm:inn}, we have  $\langle m_{\lambda},p_{(n)} \rangle = (-1)^{n-1} \epsilon_\lambda w_{(n)\lambda}.$ Clearly $w_{(n)\lambda}=1$ if and only if $\lambda=(1^n)$. Hence the theorem follows from Lemmas~\ref{lemma:gen1} and ~\ref{lemma:gen2}.
\end{proof}
Before considering the skew version of Mead's result over $\mathbb Z$, let us make the following useful definition.
\begin{definition}[Rectangular]
    We say that a partition $\lambda$ of $n$ is rectangular if $\lambda=(a^b)$ for some positive integers $a, b$ such that $ab=n$.
\end{definition}
\begin{theorem}
\label{theorem:SM}
    Let $\{\lambda^n/\mu^n\}$ be a graded skew partition sequence and for each $n\geq0$, $u_n=m_{\lambda^n/\mu^n}$. Then the set $\{u_n\}_{n\geq0}$ is algebraically independent and generates $\Lambda_\mathbb{Z}$ if and only if, for each $n\geq0$, $\lambda^n$ and $\mu^n$ is one of the following pair:
    \begin{enumerate}
        \item $\lambda^n=(1^{m+n})$ and $\mu^n=(1^m)$, 
        \item $\lambda^n=(c^d,1^n)$ with $c>n$, cd=m and $\mu^n=(1^m)$, 
        \item $\lambda^n=(1^{m+n})$ and $\mu^n=(a^b)$ with $ab=m$,
        \item $\lambda^n=(c^d)\cup(1^n)$ and $\mu^n=(a^b)$ with $\gcd(a,c)=1$ and $c>n$.
    \end{enumerate}
\end{theorem}
The proof of Theorem~\ref{theorem:SM} uses the following lemma:
\begin{lemma}
\label{lemma:SM}
    Let $\xi$ and $\lambda$ be partitions of $n$ and, $\eta$ and $\mu$ be partitions of $m$.
    Then  we have 
    \begin{enumerate}
        \item $w_{(n)\lambda}\geq n$ if $\lambda$ is non-rectangular,
        \item $w_{\xi\lambda}>1$ if $\lambda\neq (1^n)$,
        \item $w_{(\xi\cup\eta)(\lambda\cup\mu)}\geq w_{\xi\lambda} w_{\eta\mu}$.
    \end{enumerate}
\end{lemma}
\begin{proof}
 
    Let $a$ be a part of $\lambda$ occurring with multiplicity $m_a$.
    Since $\lambda$ is not rectangular, it has a part $b\neq a$.
    For each $i$ ($1\leq i\leq m_a$), construct a domino tabloid of shape $(n)$ and weight $a$ as follows: place $i$ dominoes of length $a$, followed by a domino of size $b$, and then complete the tabloid in any manner.
    This gives rise to $m_a$ distinct domino tabloids, each of weight $a$, contributing $am_a$ to $w_{(n)\lambda}$.
    Repeating this over all distinct parts of $\lambda$ will give a total contribution of $n$ to $w_{(n)\lambda}$, proving the first claim.

    For the second claim, note that $w_{\xi\lambda}=1$ if and only if $\lambda=(1^n)$.
    For if $w_{\xi\lambda}=1$, there is a unique domino tabloid of shape $\xi$, type $\lambda$ and weight $1$.
    If $\lambda$ is not a column, this tabloid must have a domino of size greater than one.
    Sliding it all the way to the left in its row will give another domino tabloid of shape $\xi$, type $\lambda$, and weight more than $1$.

     For the third claim, fix a tabloid $T$ of shape $\eta$ and type $\lambda$ we can vary all the tabloids of shape $\xi$ and type $\mu$ to get tabloids of shape $\xi\cup\eta$ and type $\mu\cup\lambda$. Now the sum of all such tabloids is $w_{\xi\mu} L$ where $L$ is the weight of the domino tabloid $T$. Now we shall vary the domino tabloid of shape $\eta$ and type $\lambda$ and all those weights then we get the required inequality. 
\end{proof}
\begin{proof}[Proof of Theorem~\ref{theorem:SM}]
Suppose $\mu\vdash m$ and $\lambda\vdash m+n$, and $\lambda$ refines $n$.
We need to determine when the sum
\begin{equation}
    \label{eq:sum}
    \sum_{\xi\vdash|\mu|} z_{\xi}^{-1} w_{\xi \mu} w_{(\xi\cup (n)) \lambda}
\end{equation}
is equal to $1$.

\subsubsection*{Case 1: $\mu$ is non-rectangular}
By Lemma~\ref{lemma:SM}, the term $\tfrac {1}{m}w_{(m)\mu} w_{(m,n)\lambda}$ in \eqref{eq:sum} corresponding to $\xi=(m)$ is greater than or equal to $1$.
If $\lambda\neq (1^{m+n})$, then $w_{(m,n)\lambda}>1$ by Lemma~\ref{lemma:SM}.
Otherwise the term in \eqref{eq:sum} corresponding to $\xi=\mu$ is positive.
In either case, the sum \eqref{eq:sum} is greater than $1.$
\subsubsection*{Case 2: $\mu=(1^m)$, $\lambda=(1^{m+n})$.} In this case, the sum~\eqref{eq:sum} becomes $\sum_{\xi\vdash m} z_{\xi}^{-1}$, which equals $1$, since $z_\xi^{-1}$ is the probability that a permutation in $S_m$ has cycle type $\xi$.
    \subsubsection*{Case 3: $\mu=(1^m)$ and $\lambda\neq (1^{m+n})$.} First, consider the case when $\lambda$ admits a non-rectangular sub-partition $\tilde\lambda$ of $m$ (here a sub-partition $\tilde\lambda$ of $\lambda$ by we mean a partition obtained from $\lambda$ by deleting some of its parts). Note that the term in ~\eqref{eq:sum} corresponding to $\xi=(m)$ is greater than or equal to $1$. 
    The term corresponding to $\xi=\tilde\lambda$ in~\eqref{eq:sum} is positive, so the sum exceeds $1$.
    Now assume that $\lambda$ admits a rectangular sub-partition $(a^b)$ of $m$, i.e., $\lambda=(a^b)\cup\eta$, where $\eta\vdash n$. We use the following identity:
        \begin{equation}
        \begin{aligned}
        \label{eq:cal-SM}
            \sum_{\xi\vdash|\mu|} z_{\xi}^{-1} w_{(\xi\cup (n))((a^b)\cup \eta)} &\geq  w_{(n)\eta} \sum_{\xi\vdash|\mu|} z_{\xi}^{-1} w_{\xi(a^b)}  \\
           & = w_{(n)\eta} \sum_{\gamma\vdash b} \dfrac{1}{\prod_{i\geq1}(ai)^{m_i}m_i!}a^{l(\gamma)}\\
           &= w_{(n)\eta} \sum_{\gamma\vdash b} \dfrac{1}{\prod_{i\geq1}i^{m_i}m_i!}\\
           &= w_{(n)\eta}
           \end{aligned}
        \end{equation}
        Note that we have used Lemma~\ref{lemma:SM} in the first inequality of the above calculation. 
        If $\eta\neq(1^n)$, then the sum~\eqref{eq:sum} is greater than $1$ by Lemma~\ref{lemma:SM}. Otherwise $\lambda=(a^b)\cup(1^n)$. Also, $a>n$ otherwise one can decompose $\lambda=\tilde\lambda\cup\eta$ where $\eta\vdash n$ and $\eta\neq(1^n)$. Thus we boil down to one of the above two cases, which results in the sum~\eqref{eq:sum} being greater than one.
        Finally, $\lambda=(a^b)\cup(1^n)$ with $a>1$. In this case, sum~\eqref{eq:sum} is 1. Because in the above calculation, the first inequality becomes equality due to $w_{(\xi\cup (n))((a^b)\cup (1^n) ) }=w_{\xi (a^b)}$.
        \subsubsection*{Case 4: $\mu=(a^b)$, with $a>1$, and $\lambda=(1^{m+n})$.}
         In this case, the sum~\eqref{eq:sum} becomes,   
        \begin{equation}
            \begin{aligned}
            \sum_{\xi\vdash|\mu|} z_{\xi}^{-1} w_{\xi\cup(a^b)}
            & =\sum_{\gamma\vdash b} \dfrac{1}{\prod_{i\geq1}(ai)^{m_i}m_i!}a^{l(\gamma)}\\
            &= \sum_{\gamma\vdash b} \dfrac{1}{\prod_{i\geq1}i^{m_i}m_i!}
            &= 1
            \end{aligned}
        \end{equation}
        \subsubsection*{Case 5: Finally, $\mu=(a^b)$ with $a>1$ and $\lambda\neq(1^{m+n})$.} Then $w_{(m)\mu}=a$. We use a case-by-case analysis:
            \begin{itemize}
                \item $\lambda$ has a non-rectangular sub-partition of $m$. Then the term $\tfrac{1}{m} a  w_{((m)\cup(n))\lambda} $ corresponding to $\xi=(m)$ in the sum \eqref{eq:sum} is greater than $1$ by Lemma~\ref{lemma:SM} and $a>1$.
                \item Suppose $\lambda$ only has rectangular sub-partitions of $m$.
                Let $(c^d)$ be a rectangular sub-partition of $m$. Let us write $\lambda=(c^d)\cup\eta$ where $\eta\vdash n$.
                Then we have 
                \begin{equation}
                \begin{aligned} \label{eq:finsum}
                    \sum_{\xi\vdash|\mu|} z_{\xi}^{-1} w_{\xi(a^b)} w_{(\xi\cup (n))((c^d)\cup \eta)} &\geq  w_{(n)\eta} \sum_{\xi\vdash|\mu|} z_{\xi}^{-1} w_{\xi(a^b)} w_{\xi(c^d)}  \\
           & = w_{(n)\eta} \sum_{\gamma\vdash g} \dfrac{1}{\prod_{i\geq1}(qi)^{m_i}m_i!}a^{l(\gamma)} c^{l(\gamma)}\\
           &= w_{(n)\eta} \sum_{\gamma\vdash g} \dfrac{1}{\prod_{i\geq1}(i)^{m_i}m_i!}\dfrac{a^{l(\gamma)} c^{l(\gamma)}}{q^{l(\gamma)}}\\
           &\geq \dfrac{ac}{q}w_{(n)\eta},
                \end{aligned}
                \end{equation}
                where $q=\mathrm{lcm}(a,c)$ and $g=\tfrac{n}{q}$. If $\eta\neq (1^n)$ or $ac=q$, then the sum \eqref{eq:sum} is greater than one by Lemma~\ref{lemma:SM} and the fact that $ac\geq q$.
                
            \item $\lambda=(c^d,1^n)$ with $\gcd(a,c)=1$.
            If $c\leq n$, then one can write $\lambda=\tilde\lambda\cup\eta$ where $\eta\vdash n$ and $\eta\neq(1^n)$.
                Thus the sum \eqref{eq:sum} is greater than one by the first two arguments of this case.
                Finally, we have $\lambda=(c^d)\cup(1^n)$ with $\ gcd(a,c)=1$ and $c>n$. In this case, the sum \eqref{eq:sum} is $1$, because all the inequalities in \eqref{eq:finsum} become equal.
            \end{itemize}
            This completes the proof.
\end{proof}
\subsection{Skew complete and skew elementary symmetric functions}
\begin{theorem}
\label{theorem:H}
Let $\{\lambda^n/\mu^n\}$ be a graded skew partition sequence and for each $n\geq0$, $u_n=h_{\lambda^n/\mu^n}$. Then the set $\{u_n\}_{n\geq0}$ is algebraically independent and generates $\Lambda_{\mathbb F}$ if and only if, for each $n\geq 1$, $\lambda_1^n\geq n$.
\begin{proof}
By Lemmas~\ref{lemma:gen1} and~\ref{lemma:gen2}, it suffices to show that $\langle h_{\lambda^n/\mu^n},p_n \rangle \neq 0$ if and only if $\lambda_1^n \geq n$.
By the definition of the adjoint of the multiplication operator, we have 
    \begin{align*}
        \langle h_{\lambda/\mu},p_n \rangle = \langle h_{\lambda},h_\mu p_n \rangle = \langle p_{n}^\perp (h_\lambda),h_\mu \rangle.
    \end{align*}
    Since $p_{n}^\perp = \sum_{j \geq 0} h_j \dfrac{\partial}{\partial h_{n+j}}$ (see \cite[Ex.~I.5.3]{MR3443860}), $p_{n}^\perp h_\lambda =0$ unless $\lambda$ contains a part which is greater than or equal to $n$.
If $\lambda$ contains a part that is greater than or equal to $n$, then $p_{n}^\perp h_\lambda$ is nonzero and it is a non-negative integer linear combination of complete symmetric functions.

If $\lambda$ and $\mu$ are partitions of $n$, then $\langle h_\lambda,h_\mu \rangle >0$, since $h_\eta=\sum_{\nu\vdash n} K_{\nu\eta}s_\nu$.
Here $K_{\nu\eta}$ is the Kostka number.
Note that $K_{(n)\eta}=1$.
Hence $\langle h_\lambda,h_\mu \rangle \geq \langle s_{(n)},s_{(n)}\rangle=1$.

From the above two observations, $\langle h_{\lambda/\mu},p_n \rangle \neq 0$ if and only if $\lambda$ contains a part  $ \geq n$. This completes the proof of the lemma.
    \end{proof} 
\end{theorem}
We shall now look at the $\mathbb Z$-generators from the skew complete symmetric functions.
\begin{theorem}
    \label{theorem:SH}
Let $\{\lambda^n/\mu^n\}$ be a graded skew partition sequence and for each $n\geq0$, $u_n=h_{\lambda^n/\mu^n}$. Then the set $\{u_n\}_{n\geq0}$ is algebraically independent and generates $\Lambda_{\mathbb{Z}}$ if and only if, for each $n\geq 1$, $\lambda_1^n\geq n$ and $(\lambda^n,\mu^n)$ satisfies one of the following conditions:
\begin{enumerate}
    \item $\mu^n=(|\lambda^n|-n)$ and $\lambda_2^n<n$,
    \item $\lambda^n=(n,m)$ with $m<n$ and $\mu^n$ can be any partition of $|\lambda^n|-n$, 
    \item $\lambda^n=(n+m)$ and $\mu^n$ can be any partition of $|\lambda^n|-n$.
\end{enumerate}
\end{theorem}
\begin{proof}
Let $\lambda\vdash m+n$, $\mu\vdash m$. We need to determine the pairs $(\lambda,\mu)$ such that $\langle h_{\lambda/\mu},p_n \rangle=1$. We can assume that $\lambda_1\geq n$ by Theorem~\ref{theorem:H}.
We have,
     \begin{align*}
       J= \langle h_{\lambda/\mu},p_n \rangle = \sum_{j \geq 0}  \left\langle  h_j \dfrac{\partial}{\partial h_{n+j}} h_\lambda, h_\mu \right\rangle.
    \end{align*}
Observe that for any two partitions $\lambda, \mu$ of $n$, $\langle h_\lambda,h_\mu\rangle$ is always a positive integer. If $\lambda$ has two parts greater than $n$, then $J>1$. Because if those two parts are distinct, then $\sum_{j\geq0}h_j \dfrac{\partial}{\partial h_{n+j}} h_\lambda \geq h_\eta+h_\gamma$ for some partitions $\eta, \gamma$ of $n$. If those two parts are equal (say, to $n+g$) then $h_g \dfrac{\partial}{\partial h_{n+g}} h_\lambda \geq 2 h_gh_{\tilde\lambda}$, where $\tilde\lambda$ is obtained from $\lambda$ by deleting a part equal to $n+g$. In either case, we have $J>1$. Now let us consider the case when $\lambda$ has exactly one part greater than $n$ and let $\sum_{j\geq0}h_j \dfrac{\partial}{\partial h_{n+j}} h_\lambda=h_{\bar\lambda}$. Note that $\langle h_\lambda,h_\mu\rangle >1$ if $\bar\lambda\neq (m)$ and $\mu \neq (m)$. Indeed,
\begin{align*}
    \langle h_{\bar\lambda},h_\mu\rangle &\geq \langle s_{(m)}+s_{(m-1,1)},s_{(m)}+s_{(m-1,1)}\rangle\\
    & = \langle s_{(m)},s_{(m)}\rangle+\langle s_{(m-1,1)},s_{(m-1,1)}\rangle=1+1=2.
\end{align*}   
Finally, we will consider the remaining possible pairs $(\lambda,\mu)$ in the following cases.
\begin{itemize}
    \item $\bar\lambda\neq(m)$ and $\mu=(m).$ In this case, $J=1$. Here $\lambda$ is any partition of $m+n$ such that exactly one part of $\lambda$ is greater than or equal to $n$, i.e., $\lambda_1^n\geq n$ and $\lambda_2^n<n$.
    \item $\bar\lambda=(m)$ and $\mu\neq (m).$ In this case also $J=1$. The only partitions $\lambda$ of $m+n$ with exactly one part is greater than or equal to $n$ such that $\bar\lambda=(m+n)$ are $(n,m)$ with $m<n$ and $(m+n)$.
\end{itemize}
This completes the proof.
\end{proof}

The above theorem holds for the skew elementary symmetric functions as well.
\begin{theorem}
\label{theorem:E}
	Let $\{\lambda^n/\mu^n\}$ be a graded skew partition sequence and for each $n\geq0$, $u_n=e_{\lambda^n/\mu^n}$. Then the set $\{u_n\}_{n\geq0}$ is algebraically independent and generates $\Lambda_{\mathbb F}$ if and only if, for each $n\geq 1$, $\lambda_1^n\geq n$.
\end{theorem}
\begin{proof}
Recall that $\Lambda_{\mathbb F}$ admits an algebra involution $\omega$ for which $\omega(h_n)=e_n$ for all $n\geq 0$, which preserves the Hall inner product.
So $\omega(h_{\lambda/\mu}) = e_{\lambda/\mu}$, and the theorem follows from Theorem~\ref{theorem:H}.
\end{proof}
The following theorem can be easily proved by using $\omega$ and Theorem~\ref{theorem:SH}.
\begin{theorem}
    \label{theorem:SE}
Let $\{\lambda^n/\mu^n\}$ be a graded skew partition sequence and for each $n\geq0$, $u_n=e_{\lambda^n/\mu^n}$. Then the set $\{u_n\}_{n\geq0}$ is algebraically independent and generates $\Lambda_{\mathbb{Z}}$ if and only if, for each $n\geq 1$, $\lambda_1^n\geq n$ and $(\lambda^n,\mu^n)$ satisfies one of the following conditions:
\begin{enumerate}
    \item $\mu^n=(|\lambda^n|-n)$ and $\lambda_2^n<n$,
    \item $\lambda^n=(n,m)$ with $m<n$ and $\mu^n$ can be any partition of $|\lambda^n|-n$, 
    \item $\lambda^n=(n+m)$ and $\mu^n$ can be any partition of $|\lambda^n|-n$.
\end{enumerate}
\end{theorem}
\subsection{Schur functions} Now we shall consider the most important family of symmetric functions, namely the Schur functions.
For the definition of terms such as Young diagram, hook, ribbon, etc., we refer the reader to \cite[Ch.~I]{MR3443860}.
\begin{theorem}
\label{theorem:theSchur}
Let $\{\lambda^n\}$ be a graded partition sequence and for each $n\geq0$, $u_n=s_{\lambda^n}$. Then the set $\{u_n\}_{n\geq0}$ is algebraically independent and generates $\Lambda_\mathbb{F}$ if and only if, for each $n\geq0$, $\lambda^n $  is a hook. 
\end{theorem}
\begin{proof}
Taking $\mu$ to be the empty partition in \cite[Ex.~I.3.11, Eq.~(2)]{MR3443860}, we have
\begin{displaymath}
    p_n = \sum_\lambda (-1)^{n-\lambda_1}s_\lambda
\end{displaymath}
where the sum varies over all hook partitions $\lambda$ of $n$.
We get
$$\langle s_{\lambda},p_n\rangle = \begin{cases}
    (-1)^{n-\lambda_1} &\text{if }\lambda\text{ is a hook,}\\
    0 &\text{otherwise.}
\end{cases}$$
Now the theorem follows from Lemmas~\ref{lemma:gen1} and~\ref{lemma:gen2}.
\end{proof}
\begin{remark}
    Since $\langle s_\lambda,p_n\rangle = \pm 1$ for every hook $\lambda$ of size $n$, Theorem~\ref{theorem:theSchur} continues to hold when $\mathbb F$ is replaced by any unital commutative ring $R$.
\end{remark}
\subsection{Skew Schur functions $s_{\lambda/\mu}$} 
\begin{theorem}\label{theorem:SkewSchur}
Let $\{\lambda^n/\mu^n\}$ be a graded skew partition sequence and for each $n\geq0$, $u_n=s_{\lambda^n/\mu^n}$. Then the set $\{u_n\}_{n\geq0}$ is algebraically independent and generates $\Lambda_{\mathbb F}$ if and only if, for each $n\geq0$, $\lambda^n/\mu^n$ is a ribbon. These also form a $\mathbb Z$-generating set for $\Lambda_\mathbb Z$.
\end{theorem}
\begin{proof}
We have 
$$\langle s_{\lambda/ \mu },p_n\rangle=\begin{cases} (-1)^{\mathrm{ht}(\lambda/ \mu )} & \text{if } \lambda/ \mu  \text{ is a ribbon of size }n,\\
0 & \text{otherwise,}
\end{cases}
$$
where $\mathrm{ht}(\lambda/ \mu )$ is the number of rows $\lambda/ \mu$ occupies minus one
(see \cite[chapter I, exercise 3.11]{MR3443860}).
Now the theorem follows from Lemmas~\ref{lemma:gen1} and ~\ref{lemma:gen2}.
\end{proof}

\section{Generators for the familiar bases over $\mathbb{F}(t)$}
\subsection{Hall-Littlewood symmetric functions and their specializations.}
The Hall-Littlewood symmetric function $P_\lambda(x;t)$ variables $x_1,\dotsc,x_n$ is defined by 
$$P_\lambda(\uline{x};t) = \dfrac{1}{v_\lambda(t)}\sum_{w\in S_n}w(x_1^{\lambda_1}x_2^{\lambda_2}\dotsb x_n^{\lambda_n}\prod_{1\leq i<j\leq n}\dfrac{x_i-tx_j}{x_i-x_j})$$
where $v_\lambda(t)=\prod\limits_{i\geq1}\prod\limits_{j=1}^{m_i}\dfrac{\phi_{m_i}(t)}{(1-t)^{m_i}}$, $m_i$ denotes the number of parts of $\lambda$ equal to $i$ , $\phi_r(t)=(1-t)(1-t^2)\dots(1-t^r)$ and $\uline{x}$ denotes $x_1,\dots,x_n$. It is well known that $P_\lambda(\uline{x};t)$ has the stability property which allows us to define the Hall-Littlewood symmetric function $P_\lambda(x;t)$ to be the inverse limit of $P_\lambda(\uline{x};t)$.

The $Q$-Hall-Littlewood symmetric functions are defined by $Q_\lambda(x;t)=(\prod\limits_{i\geq1}\phi_{m_i}(t)) P_\lambda(x;t)$

From the definition of $P_\lambda(x;t)$, we see that $P_\lambda(x;t)$ specializes to Schur functions $s_\lambda(x)$ at $t=0$ and to monomial symmetric functions $m_\lambda(x)$ at $t=1$.
It is curious to know what happens for other values of $t$. It turns out that 
\begin{theorem} \label{theorem:hallp}
Let $\{\lambda^n\}$ be a graded partition sequence and for each $n\geq0$, $u_n=P_{\lambda^n}$. Then 
 the set $\{u_n(x;t)\}_{n\geq0}$ is a system of algebraically independent generators of $\Lambda_{\mathbb{F}(t)}$.
For $\xi\in \mathbb F$, the sequence $\{u_n(x;\xi)\}$ of specializations forms a system of algebraically independent and generates $\Lambda_{\mathbf F}$ if $\xi\neq 0$ and $\xi$ is not a non-trivial root of unity.
\end{theorem} 
\begin{proof}
Recall \cite[Ch.~III,~(4.11)]{MR3443860} the $t$-inner product is given by
\begin{equation}
    \label{eq:t-inner}
    \langle p_\lambda,p_\mu\rangle_t = z_\lambda \prod_{i\geq 1}(1-t^{\lambda_i})^{-1},
\end{equation}
where $z_\lambda = \prod_{i\geq 1} i^{m_i}m_i!$.
Note that the $t$-inner product is related to the Hall inner product by (see ~\cite{MR1731818}
)\begin{equation}
    \label{eq:t-inner-prod}
    \langle f, g\rangle_t = \langle f[X/(1-t)], g(x)\rangle,
\end{equation}
where the square brackets signify plethystic substitution.

By \cite[Ex.~III.7.2]{MR3443860},
\begin{align*}
   \langle Q_\lambda(x;t),p_n \rangle_t &=t^{n(\lambda)}\phi_{l(\lambda)-1}(t^{-1})\\
   \langle P_\lambda(x;t),p_n \rangle &=\dfrac{ (1-t^n)t^{n(\lambda)}\phi_{l(\lambda)-1}(t^{-1})}{\prod_{i\geq 1}\phi_{m_i({\lambda})}(t)}, 
\end{align*}
	 where $n(\lambda)=\sum_{i\geq1}(i-1)\lambda_i $, $m_i(\lambda)$ is the number of parts of $\lambda$ equal to $i$, $l(\lambda)=\sum_i m_i$, and $\phi_r(t)=(1-t)(1-t^2)\dots(1-t^r)$.

For the second assertion of the theorem, observe that the possible roots of the above polynomial are  $n$th roots of unity and $0$ but not $1$. Note that $1$ is not a root because $1$ is a root of the denominator with multiplicity $l(\lambda)$ and also a root of the numerator with multiplicity $1+l(\lambda)-1=l(\lambda)$.
The first assertion of the theorem follows from the second.
\end{proof}
When $t=0$, $P_\lambda(x;t)$ specializes to $s_\lambda(x)$ which was considered in Theorem~\ref{theorem:theSchur}.
Let us look more closely at the case where $t$ is a root of unity.
  \begin{theorem}
 \label{theorem:halgen}
 	If $\xi_k$ is a primitive $k$th root of unity with $k>1$, then  for a partition $\lambda$ of n, $\langle P_\lambda(x;\xi_k),p_n \rangle $  is nonzero if and only if one of the following holds:
\begin{itemize}
\item $k$ divides $n$ and $$\sum_{i\geq1} \left\lfloor\dfrac{m_i(\lambda)}{k}\right\rfloor = \left\lfloor\dfrac{l(\lambda) +k-1}{k}\right\rfloor,$$
\item $k$ does not divide $n$ and $$\sum_{i\geq1} \left\lfloor\dfrac{m_i(\lambda)}{k}\right\rfloor = \left\lfloor\dfrac{l(\lambda)-1}{k}\right\rfloor.$$
\end{itemize}
Here $\lfloor x\rfloor$ denotes the greatest integer less than or equal to $x$.
Let $\{\lambda^n\}$ be a graded partition sequence and for each $n\geq0$, $u_n=P_{\lambda^n}(x;\xi_k)$. Then the set $\{u_n(x;\xi_k)\}_{n\geq0}$ is a system of algebraically independent and generates $\Lambda_{\mathbb{F}(t)}$ if and only if, for each $n\geq1$, $\lambda^n$ satisfies one of the above conditions.
\end{theorem}
\begin{proof}
Since the inner product is
\begin{equation}
    \label{eq:num-den}
    \langle P_\lambda(x;t),p_n \rangle =\dfrac{ (1-t^n)t^{n(\lambda)}\phi_{l(\lambda)-1}(t^{-1})}{\prod_{n\geq 1}\phi_{m_i({\lambda})}(t)}
\end{equation}
and that $\xi_k$ is a root of $(1-t^j)$ if and only if $k|j$. Hence $\xi_k$ occurs as a root in $\phi_r(t)$ with multiplicity $\lfloor r/k \rfloor$. Now the result easily follows by comparing the multiplicity of $\xi_k$ as a root of the numerator and denominator of \eqref{eq:num-den}.
\end{proof}
\begin{remark} \label{rem1}
    We remark that in the equalities of the statement of the Theorem~\ref{theorem:halgen}, the left-hand side is always less than or equal to the right-hand side since the inner product $\langle P_\lambda(x;t),p_n \rangle \in \mathbb F[t]$.
\end{remark}
 The following theorems are very similar to the above two theorems.
\begin{theorem} 
Let $\{\lambda^n\}$ be a graded partition sequence and for each $n\geq0$, $u_n=Q_{\lambda^n}$. Then 
 the set $\{u_n(x;t)\}_{n\geq0}$ is a system of algebraically independent generators of $\Lambda_{\mathbb{F}(t)}$.
For $\xi\in \mathbb F$, the sequence $\{u_n(x;\xi)\}_{n\geq0}$ of specializations forms a system of algebraically independent generators of $\Lambda_{\mathbb F}$ if $\xi\neq 0$ and $\xi$ is not a root of unity.
\end{theorem}
\begin{proof}
    the proof follows from the definition of $Q_\lambda(x;t)$ and the proof Theorem ~\ref{theorem:hallp}.
    The only difference is that
    \begin{equation}
    \label{eq:num-den-Q}
    \langle Q_\lambda(x;t),p_n \rangle = (1-t^n)t^{n(\lambda)}\phi_{l(\lambda)-1}(t^{-1}).
\end{equation}
    has a root at $t=1$.
\end{proof}
When $t=0$, $Q_\lambda=s_\lambda$, we are reduced to the case of Schur functions (Theorem~\ref{theorem:theSchur}).
When $t$ is a root of unity, we have the following result.
 \begin{theorem}
 \label{theorem:halgenQ}
 	If $\xi_k$ is a primitive $k$th root of unity, then for a partition $\lambda$ of n, $\langle Q_\lambda(x;\xi_k),p_n \rangle $  is nonzero if and only if $k$ does not divide $n$ and $k\geq l(\lambda)-1$.
    In particular, for any graded partition sequence $\{\lambda^n\}$, $u_n=Q_{\lambda^n}(x;\xi_k)$, $\{u_n\}_{n\geq0}$ is algebraically dependent.  
\end{theorem}
\begin{proof}
By \eqref{eq:num-den-Q}, $\xi_k$ is a root of $\langle Q_\lambda(x,\xi_k),p_n\rangle$ if and only if $k|n$ or $k\leq l(\lambda)-1$.
Thus, for any fixed $k$, $\langle Q_\lambda(x,\xi_k),p_n\rangle=0$ when $n=k$ and $\lambda\vdash n$.
Hence $\{Q_{\lambda^n}(x,\xi_k)\}_{n\geq0}$ will not be algebraically independent by Lemmas~\ref{lemma:gen1} and ~\ref{lemma:gen2}.
\end{proof}
\subsection{Big Schur functions $S_{\lambda}(x;t)$ and the functions $q_{\lambda}(x;t)$} 
For each $n\geq 0$, let $q_n(x;t)=Q_{(n)}(x;t)$.
The big Schur function is defined by $S_\lambda(x;t)=\det(q_{\lambda_i-i+j}(x;t))$.

The big Schur function $S_{\lambda}(x;t)$ behaves like the Schur function $s_{\lambda}(x)$.
Indeed, we have
\begin{theorem}
Let $\{\lambda^n\}$ be a graded partition sequence and for each $n\geq0$, $u_n=S_{\lambda^n}(x;t)$. Then the set $\{u_n\}_{n\geq0}$ is algebraically independent and generates $\Lambda_{\mathbb{F}(t)}$ if and only if, for each $n\geq0$, $\lambda^n $  is a hook.
For $\xi\in \mathbb F$, the sequence $\{u_n(x;\xi)\}_{n\geq0}$ of specializations forms a system of algebraically independent generators of $\Lambda_{\mathbb F}$ if and only if $\lambda^n$ is a hook for each $n\geq 0$ and $\xi$ is not a root of unity.
\end{theorem}
\begin{proof}
Since $\{S_\lambda\mid \lambda\in\Par\}$ and $\{s_\lambda\mid\lambda\in\Par\}$ are dual bases with respect to the $t$-inner product and using the equation~\eqref{eq:t-inner-prod}, we have
\begin{align*}
    \langle S_\lambda(x;t),p_n \rangle & =  (1-t^n)\langle s_\lambda(x),p_n \rangle\\
    & = \begin{cases}
        (-1)^{n-\lambda_1}(1-t^n)&\text{if }\lambda\vdash n \text{ is a hook,}\\
        0 & \text{otherwise,}
    \end{cases}
\end{align*}
from which the theorem follows.
\end{proof}
Since, $q_{\lambda}(x;t)=\prod\limits_{i\geq1}q_{\lambda_i}(x;t)$, the only algebraically independent generating set that can come from $\{q_{\lambda}(x;t)\mid \lambda \in \Par\}$ is $\{q_n(x;t)\}_{n\geq0}$.
\subsection{Skew Hall-Littlewood symmetric functions}
The skew Hall-Littlewood symmetric function $P_{\lambda/\mu}$ is defined by $\langle P_{\lambda/\mu},P_\nu\rangle_t =\langle P_{\lambda},P_\mu P_\nu\rangle_t$.
Since at $t=0$, $P_\lambda(x;0)=s_\lambda(x)$ and the $t$-inner product coincides with our usual inner product, we have $P_{\lambda/\mu}(x;0) =s_{\lambda/\mu}(x).$ Using Theorem~\ref{theorem:SkewSchur}, we have
\begin{theorem}
Let $\{\lambda^n/\mu^n\}$ be a graded skew partition sequence and for each $n\geq0$, $u_n=P_{\lambda^n/\mu^n}$. Then the set $\{u_n\}_{n\geq0}$ is algebraically independent and generates $\Lambda_{\mathbb F}$ if, for each $n\geq0$, $\lambda^n/\mu^n$ is a ribbon.
\end{theorem}
The converse seems to be an interesting open problem. In fact, for any partition $\lambda$, finding the P-Hall-Littlewood expansion of $P_\lambda(x;t) p_n$ is interesting in its own right.
\begin{conjecture}
    Let $\{\lambda^n/\mu^n\}$ be a graded skew partition sequence and for each $n\geq0$, $u_n=P_{\lambda^n/\mu^n}$. The set $\{u_n\}_{n\geq0}$ is algebraically independent and generates $\Lambda_{\mathbb{F}(t)}$ then for each $n\geq0$, $\mu^n\subset \lambda^n$ and the skew diagram $\lambda^n/\mu^n$ cannot be separated by a column i.e.,
    there exists no $j$ such that the $j$th column of $\lambda^n/\mu^n$ has no cells, but $\lambda^n/\mu^n$ has cells both to the left and to the right of the $j$th column.
\end{conjecture}
\subsection{Schur's $P$-functions and $Q$-functions}
Schur's $P$-functions and $Q$-functions are specializations of $P_\lambda(x;t)$ and $Q_\lambda(x;t)$ at $t=-1$ respectively. We shall use the general result in Theorem~\ref{theorem:halgen} for this specialization. So we easily deduce the following from Theorem~\ref{theorem:halgen}.
\begin{theorem}
\label{theorem:Schur-P}
For a partition $\lambda$ of $n$, $\langle P_\lambda(x;-1),p_n \rangle $  is nonzero if and only if one of the following holds:
\begin{itemize}
\item $2$ divides $n$ and $\sum_{i\geq1} \lfloor(m_i(\lambda)/2)\rfloor = \lfloor(l(\lambda) +1)/2\rfloor$,
\item $2$ does not divide $n$ and $\sum_{i\geq1} \lfloor(m_i(\lambda)/2)\rfloor = \lfloor(l(\lambda) -1)/2\rfloor$.
\end{itemize}
Let $\{\lambda^n\}$ be a graded partition sequence and for each $n\geq0$, $u_n=P_{\lambda^n}(x;-1)$. Then the set $\{u_n(x;-1)\}_{n\geq0}$ is a system of algebraically independent generators of $\Lambda_{\mathbb{F}}$ if and only if, for each $n\geq1$, $\lambda^n$ satisfies one of the above properties.
\end{theorem}
Since $Q_\lambda(x;-1)=\begin{cases}2^{l(\lambda)} P_\lambda(x;-1)\quad \text{ if } \lambda \text{ consists of distinct parts,}\\
0 \quad \quad \text{    Otherwise.}
\end{cases}$
So the Theorem~\ref{theorem:Schur-P} implies the following.
\begin{theorem}
For a partition $\lambda$ of $n$, $\langle Q_\lambda(x;-1),p_n \rangle $  is nonzero if and only if 
$n$ is odd and $2>l(\lambda)-1$. In particular, for any graded partition sequence $\{\lambda^n\}$ with $u_n=Q_{\lambda^n}(x;-1)$, $\{u_n\}_{n\geq0}$ is algebraically dependent.  
\end{theorem}
\section{Generators for the familiar bases over $\mathbb{F}(q,t)$}
\subsection{The Macdonald symmetric function $P_{\lambda}(x;q,t)$} The Macdonald symmetric function can be defined in many ways, we refer the reader to \cite{MR3443860}.
By \cite[Ex.~VI.8.8]{MR3443860}, we have 
\begin{equation}
\label{eq:mac} \langle P_\lambda(x;q,t),p_n \rangle =\dfrac{(1-t^n)X_{n}^{\lambda}(q,t)}{ c_{\lambda}(q,t)}
\end{equation}
where $$X_n^{\lambda}(q,t)=\prod_{(i,j)}(t^{(i-1)}-q^{(j-1)})$$ the product varies over the cells $(i,j) \in \lambda $ with the exception of $(1,1)$ and $$c_{\lambda}(q,t)=\prod_{s\in \lambda}(1-q^{a(s)}t^{l(s)+1}).$$
Hence we can easily give a sufficient condition for a sequence of Macdonald symmetric functions to form a system of algebraically independent generators for $\Lambda_{\mathbb F(q,t)}$ from the above equation.
\begin{theorem} Let for each $n\geq0$, $u_n=P_\lambda(x;q,t)$ for some  partition $\lambda$ of $n$. Then the set $\{u_n\}_{n\geq0}$ is algebraically independent and generates $\Lambda_{\mathbb{F}(q,t)}$. 
 For $\xi,\eta\in \mathbb{F}(q,t)$, the sequence $\{u_n(x;\xi,\eta)\}$ of specializations forms a system of algebraically independent generators of $\Lambda_{\mathbb{F}(t,q)}$ if $\eta$ not a root of unity and $\xi^i\neq \eta^j$ for any integers $i$ and $j$. 
\end{theorem}
\subsection{$Q$-Whittaker symmetric functions $W_\lambda(x;q)$} 
The $Q$-Whittaker symmetric functions $W_\lambda(x;q)$ can be obtained by substituting $t=0$ in the Macdonald symmetric function $P_\lambda(x;q,t)$, i.e., $W_\lambda(x;q)= P_\lambda(x;q,0).$
\begin{theorem}
Let $\{\lambda^n\}$ be a graded partition sequence and for each $n\geq0$, $u_n=W_{\lambda^n}(x;q)$. Then $\{u_n\}_{n\geq0}$ is algebraically independent and generates $\Lambda_{Q(q)}$. For $\xi\in \mathbb F$, the sequence $\{u_n(x;\xi)\}_{n\geq0}$ of specializations forms a system of algebraically independent generators of $\Lambda_{\mathbb{F}(t)}$ if $\xi\neq 0$ and $\xi$ is not a root of unity.
\end{theorem}
\begin{proof}
Let $\lambda$ be a partition of $n$. Since we know that the inner product $ \langle P_\lambda(x;q,t),p_n \rangle$ ~\eqref{eq:mac}. Substituting $t=0$, we obtain 
\begin{equation}
   \langle W_\lambda(x;q),p_n \rangle = (-1)^{|\lambda|-{\lambda}_{1}} q^{n(\lambda') - \lambda_1 (\lambda_1 -1)/2 } \prod_{i=1}^{\lambda_1 -1}(1-q^i) 
\end{equation}
where $n(\lambda')=\sum_{i\geq1}(i-1)\lambda_i' $.

Hence, from the above equation, the inner product is nonzero and the only roots are zero and roots of unity. Hence the theorem follows by lemmas \ref{lemma:gen1},\ref{lemma:gen2}. 
\end{proof}
If we specialize $q=0$ in $W_\lambda(x;q)$, then we get the Schur function $s_\lambda$ in which case we know what the generating sequences are~\ref{theorem:theSchur}. 
Let us look more closely at the case where $q$ is a root of unity. From the above inner product, we can make the following theorem.
 
\begin{theorem}
If $\xi_k$ is a primitive root of unity, then  for a partition $\lambda$ of $n$, $\langle W_\lambda(x;\xi_k),p_n \rangle $  is nonzero if and only if $\lambda_1^n\leq k$. 

Let $\{\lambda^n\}$ be a graded partition sequence and for each $n\geq0$, $u_n=W_{\lambda^n}(x;\xi_k).$ Then $\{u_n\}_{n\geq0}$ forms an algebraically independent set if and only if, for each $n\geq1$, $\lambda_1^n\leq k$.
\begin{proof}
The result follows from the above inner product and the Lemmas~\ref{lemma:gen1},~\ref{lemma:gen2}.
\end{proof}
\end{theorem}

\subsection{ Integral form of the Macdonald symmetric function $J_{\lambda}(x;q,t)$} The integral form of the Macdonald symmetric function corresponding to the partition $\lambda$ is defined by $J_{\lambda}(x;q,t)=c_{\lambda}(q,t) P_{\lambda}(x;q,t)$ where $P_{\lambda}(x;q,t)$ is the Macdonald symmetric function corresponding to the partition $\lambda$. From the definition of $J_\lambda(x;q,t)$ and Eq\eqref{eq:mac} we have,
\begin{align*}
    \langle J_\lambda(x;q,t),p_n \rangle =(1-t^n)X_{n}^{\lambda}(q,t),
\end{align*} 
where $X_n^{\lambda}(q,t)=\prod_{(i,j)}(t^{(i-1)}-q^{(j-1)})$ the product varies over the cells $(i,j) \in \lambda $ with the exception of $(1,1)$.  
Hence we have
\begin{theorem} For each $n\geq0$, let $u_n=J_\lambda(x;q,t)$ for some  partition $\lambda$ of $n$. Then the set $\{u_n\}_{n\geq0}$ is algebraically independent and generates $\Lambda_{\mathbb{F}(q,t)}$. For $\xi,\eta\in \mathbb{F}(q,t)$, the sequence $\{u_n(x;\xi,\eta)\}_{n\geq0}$ of specializations forms a system of algebraically independent generators of $\Lambda_{\mathbb{F}(t,q)}$ if $\eta$ is not a root of unity and $\xi^i\neq \eta^j$ for any nonnegative integers $i$ and $j$. 
\end{theorem}
\subsection*{Acknowledgements}
I thank my advisor, Amritanshu Prasad, for his invaluable support and helpful discussions. I thank Aritra Bhattacharya, V. Sathish Kumar, and Sankaran Viswanath for fruitful discussions and encouragement. I also thank Arvind Ayyer, K.N. Raghavan, and Arun Ram for their encouragement.

\bibliographystyle{abbrv}
\bibliography{refs}
\end{document}